\documentclass{amsart}

 \usepackage{amssymb,amsfonts,amsthm,amscd,stmaryrd}
\usepackage[foot]{amsaddr}
\usepackage[dvipsnames]{xcolor} 

\usepackage{pgf,tikz}
 \usetikzlibrary{arrows}

\usepackage{ae}
\usepackage[utf8]{inputenc}
\numberwithin{figure}{section}

\usepackage{mathrsfs,a4wide}
\usepackage{esint}

\usepackage{color}

\usepackage{import}

\usepackage[leqno]{amsmath}

\usepackage[final,allcolors=blue,colorlinks=true]{hyperref}

\numberwithin{figure}{section}

\newtheorem{thm}{Theorem}[section]
\newtheorem{lem}[thm]{Lemma}
\newtheorem{prop}[thm]{Proposition}

\theoremstyle{definition}
\newtheorem{defn}[thm]{Definition}

\numberwithin{equation}{section}

\def\Chi#1{\hbox{{\large $\chi$}{\Large $_{_{#1}}$}}}

\begin{document}
\def\R{\mathbb R}
\def\H{\mathcal H}
\def\S{\mathbb S}
\def\e{\varepsilon}
\def\pa{\partial}
\def\de{\partial}
\def\ds{\displaystyle}
\def\restrict#1{\raise-.5ex\hbox{\ensuremath|}_{#1}}

\title[\title{ A quantitative Weinstock inequality }]{A quantitative Weinstock inequality}

\author[N. Gavitone, D. A. La Manna, G. Paoli, L. Trani]{
  Nunzia Gavitone, Domenico Angelo La Manna, Gloria Paoli, Leonardo Trani}
 \address{
Nunzia Gavitone \\
Universit\`a degli studi di Napoli ``Federico II''\\
Dipartimento di Ma\-te\-ma\-ti\-ca e Applicazioni ``R. Caccioppoli''\\
Complesso di Monte Sant'Angelo, Via Cintia,
80126 Napoli, Italia.
}
\email{nunzia.gavitone@unina.it}
 \address{Domenico Angelo La Manna\\
 Universit\`a degli Studi di Cassino
e del Lazio Meridionale \\
Dipartimento di Ingegneria Elettrica e dell'Informazione \\
Via G. Di Biasio 43 \\
03043 Cassino, Italia}
\email{domenicolamanna@hotmail.it}
       \address{ Gloria Paoli \\ Universit\`a degli studi di Napoli Federico II, Dipartimento di Matematica e Applicazioni ``R. Caccioppoli'', Via Cintia, Monte S. Angelo - 80126 Napoli, Italia}
      \email{gloria.paoli@unina.it}
       \address{Leonardo Trani, Universit\`a degli studi di Napoli Federico II, Dipartimento di Matematica e Applicazioni ``R. Caccioppoli'', Via Cintia, Monte S. Angelo - 80126 Napoli, Italia}
      \email{leonardo.trani@unina.it}

\maketitle

\begin{abstract}
This  paper is devoted to the study of a quantitative Weinstock inequality in higher dimension for the first non trivial Steklov  eigenvalue of the  Laplace operator for convex sets. The key role is played by a quantitative  isoperimetric inequality which involves  the boundary momentum, the volume and the perimeter of a convex open set of $\R^n$, $n \ge 2$.  
\vspace{.2cm}

\end{abstract}

\section{Introduction}
Let $\Omega \subset \R^n$, with $n \ge 2$,  be a bounded, connected, open set  with Lipschitz boundary. In this paper we consider the following Steklov eigenvalue problem for the Laplace operator:
\begin{equation}\label{stekintro}
\begin{cases}
\Delta u=0 &\mbox{in}\ \Omega,\\[.2cm]
\frac{\de u}{\de \nu}=\sigma u&\mbox{on}\ \de\Omega,
\end{cases}
\end{equation}
where  by $\partial u/\partial \nu$ we denote  the outer normal derivative to $u$ on $\partial\Omega$.
It is well-known (see for instance \cite{b,h,bdr}) that the spectrum is discrete; as a consequence, we have that  there exists a sequence of eigenvalues, $0=\sigma_0(\Omega)<\sigma_1(\Omega)\leq\sigma_2(\Omega)\leq\dots\nearrow+\infty$, called the Steklov eigenvalues of $\Omega$. 
In particular, the first non trivial Steklov eigenvalue of $\Omega$ has the following variational characterization:
\begin{equation}\label{first_stec}
\sigma_1(\Omega)=\min\left\{\dfrac{\ds\int_{\Omega}|\nabla v|^2\;dx}{\ds\int_{\de\Omega}v^2\;d \mathcal{H}^{n-1}} \; :\; v\in H^1(\Omega)\setminus\{0\},\;\int_{\de\Omega}v\;d \mathcal{H}^{n-1}=0    \right\},
\end{equation}
where $\mathcal{H}^{n-1}$ denotes the $(n-1)-$dimensional Hausdorff measure in $\mathbb{R}^n$.
If we take $\Omega=B_{R}(x)$, where $B_R(x)$ is a ball  of radius $R$ cemtered at the point $x$, then
\begin{equation}
\label{ball}
\sigma_1(B_{R}(x))=\frac{1}{R}.
\end{equation}
Moreover, we know that $\sigma_1(B_{R}(x))$ has multiplicity $n$ and the corresponding eigenfunctions are $ u_i (x) = x_{i-1}$, with $i=2, \ldots,n+1$.
In \cite{We,We2}  the author considers the  problem of maximizing $\sigma_1(\Omega)$  in the plane, keeping fixed the perimeter of $\Omega$. More precisely, if $\Omega \subset \R^2$ is a simply connected and open set,  the following inequality, known as Weinstock inequality, is proved 
\begin{equation}
\label{wi}
\sigma_1(\Omega) P(\Omega) \le \sigma_1(B_R(x)) P(B_R(x)),
\end{equation}
where $P(\Omega)$ denotes the Euclidean perimeter of $\Omega$. Inequality \eqref{wi}  states that, among all  planar,  simply connected,  open sets  with prescribed perimeter, $\sigma_1(\Omega) $ is maximum for the disk.
In  \cite{bfnt}, the authors generalize the Weinstock inequality \eqref{wi} in any dimension, when restricting to the   class of convex sets. More precisely, if $\Omega\subset \R^n$ is an open, bounded, convex set, then
\begin{equation}
\label{wicri}
\sigma_1(\Omega) P(\Omega)^{\frac{1}{n-1}} \le \sigma_1(B_R(x)) P(B_R(x))^{\frac{1}{n-1}}
\end{equation}
and equality holds only if $\Omega$ is a ball.
In \cite{br} the author investigated  and solved the problem of maximizing $\sigma_{1}(\Omega)$ keeping the volume fixed. It is proved that 
\begin{equation}
\label{bwi}
\sigma_1(\Omega) V(\Omega)^{\frac 1 n} \le \sigma_1(B_R(x)) V(B_R(x))^{\frac 1 n},
\end{equation}
where $V(\Omega)$ denotes the Lebesgue measure of $\Omega$.
 Recently, in \cite{bdr},  a quantitative version of  inequality \eqref{bwi} has been   proved.

The aim of this paper is to prove a quantitative  version  of inequality \eqref{wicri}. Let $\omega_n$ be the measure of the $n$-dimensional unit ball in $\R^{n}$ and let  $d_\H$ be the Hausdorff distance (defined in \eqref{disth}). We consider the following asymmetry functional
\begin{equation}
\label{asyi}
\mathcal{A_\mathcal{H}}(\Omega)=\min_{x \in \R^{n}}\left\{\left(\frac{d_\H \left(\Omega,B_{R}(x)\right)}
{R}\right), \, P(B_{R}(x)) =P(\Omega)\right\},
\end{equation}
where $\Omega \subset \R^n$ is a bounded,  open, convex set.  We observe that $\mathcal{A_\mathcal{H}}(\Omega)$ is scaling invariant, hence
\[
\mathcal{A_\mathcal{H}}(\Omega)= \mathcal{A_\mathcal{H}}(F), 
\]
where $F$ is a convex set having the same perimeter of the unit ball, that is $P(F)=n\omega_n$.
Our main result is stated in the following theorem.
\begin{thm}
\label{main_w} Let $n \ge 2$.  There exist two costants $\bar\delta>0$ and  $C=C(n)>0$ such that for every $\Omega \subset \mathbb{R}^n$  bounded,  convex open  set with  $\sigma_{1}(B_{R}(x))\le(1+\bar\delta)\,\sigma_{1}(\Omega)$, where $B_{R}(x)$ is a ball with  $P(B_R(x))=P(\Omega)$, then \begin{equation}
\label{th_{intro}}
\dfrac{\sigma_{1}(B_{R}(x))-\sigma_{1}(\Omega)}{\sigma_{1}(\Omega)} \ge \begin{cases}
C \left( \displaystyle \mathcal{A_\mathcal{H}}(\Omega) \right)^{\frac{5}{2}} &\text{ if } n = 2
\\
C\,g\!\left(\left( \displaystyle \mathcal{A_\mathcal{H}}(\Omega) \right)^2\right)
 &\text{ if } n = 3\\
 C\left( \displaystyle \mathcal{A_\mathcal{H}}(\Omega) \right)^{\frac{n+1}{2}} &\text{ if } n \ge 4,
 \end{cases}
\end{equation}
where $g$ is the inverse function of $f(t)=t\log\left(\frac{1}{t}\right)$, for $0<t<e^{-1}$.
\end{thm}
The key point to prove Theorem \ref{main_w} is a quantitative version of a weighted isoperimetric inequality (see Theorem \ref{main2} for the precise statement).

\section{Notation and Preliminary results}

\subsection{Notation and some definitions}
Throughout the paper,  the unit ball centered at the origin will be denoted by $B$ and its boundary by $\S^{n-1}$; moreover, we will denote by $B_R$ the ball  centered at the origin of radius $R$ and by $B_R(x)$ the  ball centered at $x$ of  radius $R$.

Let $\Omega\subseteq\mathbb{R}^n$ be a bounded, open set and let $E\subseteq\R^{n}$ be a measurable set. For the sake of completeness, we recall here the definition of the perimeter of $E$ in $\Omega$:
\begin{equation*}
P(E;\Omega)=\sup\left\{  \int_E {\rm div} \varphi\:dx :\;\varphi\in C^{\infty}_c(\Omega;\mathbb{R}^n),\;||\varphi||_{\infty}\leq 1 \right\}.
\end{equation*}
The perimeter of $E$ in $\mathbb{R}^n$ will be denoted by $P(E)$ and, if $P(E)<\infty$, we say that $E$ is a set of finite perimeter.  Moreover, if $E$ has Lipschitz boundary, we have that \begin{equation}
\label{per}
 P(E)=\mathcal{H}^{n-1}(\de E),
 \end{equation}
where  $\mathcal{H}^{n-1}$ is the $(n-1)-$dimensional Hausdorff measure in $\mathbb{R}^n$.

We denote by 
\begin{equation}
\label{vol}
V(E)=\int_{E} dx
 \end{equation}
the volume of the measurable set  $E\subseteq\mathbb{R}^n$, i.e. its $n$-dimensional Lebesgue measure and, if $E$ has Lipschitz boundary, we denote by
\begin{equation}
\label{momento} 
 W(E)=\int_{\de E} |x|^2\; d\mathcal{H}^{n-1}
\end{equation} 
 the boundary momentum of $E$, where $|\cdot|$ is  the euclidean norm in $\mathbb{R}^n$.
We observe that $P(E)$, $W(E)$ and $V(E)$ have the following scaling properties, for $t>0$,
\begin{equation}
\label{res_prop}
P(tE)= t^{n-1}P(E)\qquad
V(tE))=t^n V(E)\qquad
W(tE)=t^{n+1} W(E).
\end{equation}
Finally, we recall the definition of  Hausdorff distance between two non-empty compact sets $E,F\subset\mathbb{R}^n$, that is (see for instance \cite{sch}):
\begin{equation}
\label{disth}
 d_{\mathcal H}(E,F)=\inf \left\{  \varepsilon>0  \; :\; E\subset F+B_{\varepsilon},\; F\subset E+B_{\varepsilon}\right\} .
 \end{equation}
 Note that, in the case $E$ and $F$ are  convex sets, we have  $d_\H(E,F)=d_\H (\pa E, \pa F)$ and the following rescaling property holds
 \[
 d_\H(tE,tF)=t\,d_\H (E,F), \quad t>0.
 \]
Let $\Omega \subset \R^n$ be  a bounded, open, convex set. We consider the following asymmetry functional related to $\Omega$:
\begin{equation}
\label{asy}
\mathcal{A_\mathcal{H}}(\Omega)=\min_{x \in \R^{n}}\left\{\left(\frac{d_\H \left(\Omega,B_{R}(x)\right)}
{R}\right), \,  P(\Omega)=P(B_{R}(x))\right\},
\end{equation}
 \begin{defn}
	Let $\Omega \subseteq \R^{n}$ be a bounded, open set, let $(E_j) \subset \R^{n}$ be a sequence of measurable sets and let $E\subset \R^{n}$ be a measurable set. We say that $(E_j)$ converges in measure in $\Omega$ to $E$, and we write $E_j\rightarrow E$, if $\chi_{E_j}\rightarrow\chi_E$ in $L^1(\Omega)$, or in other words, if $\lim_{j\to \infty}V(\left(E_j\Delta E\right)\cap\Omega)=0$.
\end{defn}
We recall also that the perimeter is lower semicontinous with respect to the local convergence in measure, that means, if the sequence of sets $\left(E_j\right)$ converges in measure in $\Omega$  to $E$, then
$$ P(E;\Omega)\leq \liminf_{j\rightarrow\infty}P(E_j;\Omega).$$
As a consequence of the Rellich-Kondrachov theorem,  the following compactness result holds; for a reference  see for instance \cite{afp}.
\begin{prop}
	Let $\Omega \subseteq \R^{n}$ be a bounded,  open set and  let $\left(E_j\right)$ be a sequence of measurable sets of $\mathbb{R}^n$, such that $\sup_j P(E_j;\Omega)<\infty$. Then, there exists  a subsequence $\left(E_{j_{k}}\right)$ converging in measure in $\Omega$ to a  set $E$, such  that $$ P(E;\Omega)\leq \liminf_{k\rightarrow\infty}P(E_{j_k};\Omega).$$
\end{prop}
Another useful property concerning the sets of finite perimeter is stated in the next approximation result.
\begin{prop}
	Let $\Omega \subseteq \R^{n}$ be a bounded,  open set and let $E$ be a set of finite perimeter in $\Omega$. Then, there exists a sequence of smooth, bounded open sets $\left(E_j\right)$ converging in measure in $\Omega$ and such that $\lim_{j\to \infty}P(E_j;\Omega) = P(E;\Omega)$. 
\end{prop}

In the particular case of convex sets, the following lemma holds.
\begin{lem} \label{periconv}
	Let  $(E_j) \subseteq \R^{n}$ be  a sequence of  convex sets such that  $E_j\rightarrow B$ in measure, then $\lim_{j\to \infty}P(E_j)= P(B)$.
\end{lem}
\begin{proof}
	Since, in the case of  convex sets,  the convergence in measure implies the Hausdorff convergence, we have that $\lim_{j\to \infty}d_{\mathcal H}(E_j,B)= 0$ (see for instance \cite{eft}). Thus, for $j$ large enough, there exists $\varepsilon_j$, such that 
	$$
	(1-\varepsilon_j)E_j \subset B \subset (1+\varepsilon_j)E_j .
	$$
	Being the perimeter monotone with respect to the inclusion of convex sets  then
	$$
	(1-\varepsilon_j)^{n-1} P(E_j) \leq P(B)\leq (1+\varepsilon_j)^{n-1}P(E_j).
	$$
When  $j$ goes to infinity, we have the thesis.
\end{proof}

We conclude this paragraph by recalling the following result (see \cite{eft}).
\begin{lem}
Let $K\subseteq\mathbb{R}^n$, $n\geq 2$, be a bounded, open,  convex set. There exists a positive constant $C(n)$ such that\begin{equation}
\label{diam}
	{\rm diam}(K)\leq C(n)\dfrac{P(K)^{n-1}}{V(K)^{n-2}}.
\end{equation}
\end{lem}
\subsection{Nearly spherical sets}
In this section we  give the definition of nearly spherical sets  and we recall some of  their basic properties (see for instance \cite{bdp,f, fusco}).
\begin{defn} \label{nearlysp}
Let $n\geq 2$. An open, bounded  set $E\subset \R^n$ is said  a nearly spherical set  parametrized by $v$,  if  there exists $v \in W^{1,\infty}(\S^{n-1})$ such that 
\begin{equation}\label{nearly}
    \pa E= \left\{y \in \R^n \colon y=x(1+v(x)), \, x \in \S^{n-1}\right\},
\end{equation}
with $||v||_{W^{1,\infty}}\le\frac{1}{2}$.
\end{defn}
	Note also that $||v||_{L^\infty}=d_{\mathcal H}(E,B)$. 	
The  perimeter,  the volume and the boundary momentum of a nearly spherical set are given by	
	 \begin{equation}\label{perimeter_ns}
	P(E)=\int_{\mathbb{S}^{n-1}} \left(1+v(x)\right)^{n-2}\sqrt{\left(1+v(x)\right)^2+|D_{\tau}v(x)|^2}\;d\mathcal{H}^{n-1},
	\end{equation}
	\begin{equation}\label{vol_ns}
	V(E)=\dfrac{1}{n}\int_{\mathbb{S}^{n-1}} \left(1+v(x) \right)^n\;d\mathcal{H}^{n-1},
	\end{equation}
	\begin{equation}\label{momentum_ns}
	W(E)=\int_{\mathbb{S}^{n-1}} \left(1+v(x)\right)^{n}\sqrt{\left(1+v(x)\right)^2+|D_{\tau}v(x)|^2}\;d\mathcal{H}^{n-1}.
	\end{equation}

Finally,  we recall  two lemmas that we will use later.  The first one is an interpolation result;  for its proof we refer for instance to \cite{f,fusco}. 
\begin{lem} \label{interp}
If $v \in W^{1,\infty}(\S^{n-1})$ and $\displaystyle\int_{\S^{n-1}}v \,d\H^{n-1}=0$, then
\begin{eqnarray}
||v||_{L^{\infty}(\S^{n-1})}^{n-1} \le 
\begin{cases}
\pi \|D_\tau v\|_{L^2(\S^{n-1})} & n=2\\
4||D_\tau v||^2_{L^{2}(\S^{n-1})} \log \frac{8e ||D_\tau v||_{L^{\infty}(\S^{n-1})}^{n-1}}{||D_\tau v||_{L^{2}(\S^{n-1})}^{2}} & n=3
\\[.2cm]
C(n)  ||D_\tau v||_{L^{2}(\S^{n-1})}^{2} ||D_\tau v||_{L^{\infty}(\S^{n-1})}^{n-3}\,\, & n=4
\end{cases}
\end{eqnarray}
\end{lem}
For this second lemma see for instance   \cite{fusco}.
\begin{lem} \label{utile}
Let $n\geq 2$. There exists $\varepsilon_0$ such that, if $E$ is a convex, nearly spherical set with $V(E)=V(B)$ and $||v||_{W^{1,\infty}} \le \varepsilon_0$, then
\begin{equation} \label{utileq}
    ||D_\tau v||_{L^{\infty}} ^2 \le 8||v||_{L^{\infty}}.
\end{equation}
\end{lem}
Finally, we prove the following 
\begin{lem}
Let $n\geq 2$ and let $E \subseteq \R^n$ be a  bounded, convex, nearly spherical set with 
 $||v||_{W^{1,\infty} }\le \delta$, then
\begin{equation}
\label{distsharp}
d_{\mathcal H}(E,E^*) \le C(n) d_{\mathcal H}(E,E^{\sharp}),
\end{equation}
where $E^*$ and $E^{\sharp}$ are  the balls centered at the origin  having, respectively, the same perimeter and the same volume as $E$.
\end{lem}
\begin{proof}
By the properties of the Hausdorff distance,  we get
\begin{multline}
\label{primo}
d_{\mathcal H}(E,E^*) \le d_{\mathcal H}(E,E^{\sharp}) + d_{\mathcal H}(E^*,E^{\sharp})=d_{\mathcal H}(E,E^{\sharp})+ \left(\dfrac{P(E)}{n\omega_n}\right)^{\frac {1}{n-1}}-\left(\dfrac{V(E)}{\omega_n}\right)^{\frac {1}{n}}\\=d_{\mathcal H}(E,E^{\sharp})+\left(\dfrac{V(E)}{\omega_n}\right)^{\frac {1}{n}} \left[\left(\dfrac{P(E)}{n\omega_n^{\frac 1 n}V(E)^{\frac{n-1}{n}}}\right)^{\frac {1}{n-1}}-1\right].
\end{multline}
We stress  that, in the square brackets,  we have the isoperimetric deficit of $E$, which is scaling invariant. Let  $F \subseteq \R^n$ be a convex, nearly spherical set  parametrized by $v_F$,  with 
 $||v_F||_{W^{1,\infty}} \le \delta$ and $V(F)=V(B)$. Being $F$ nearly spherical and $||v_F||_{W^{1,\infty}} \le \delta$, from \eqref{perimeter_ns} and Lemma \ref*{utile},  we get
\begin{multline}
\left(\dfrac{P(F)}{n\omega_n^{\frac 1 n}V(F)^{\frac{n-1}{n}}}\right)^{\frac {1}{n-1}}-1=\left(\dfrac{P(F)}{n\omega_n}\right)^{\frac {1}{n-1}}-1=\\=\left(\frac{1}{n \omega_n}\int_{\mathbb{S}^{n-1}} \left(1+v_F(x)\right)^{n-2}\sqrt{\left(1+v_F(x)\right)^2+|D_{\tau}v_F(x)|^2}\, \right)^{\frac {1}{n-1}}-1 \le \\  \le C(n) ||v_F||^{2}_{W^{1,\infty}} \le C(n)||v_F||_{L^{\infty}}. 
\end{multline}
As a consequence,  recalling that $||v_F||_{L^\infty}=d_{\mathcal H}(F,B)$,
\[
\left(\dfrac{V(E)}{\omega_n}\right)^{\frac {1}{n}} \left[\left(\dfrac{P(E)}{n\omega_n^{\frac 1 n}V(E)^{\frac{n-1}{n}}}\right)^{\frac {1}{n-1}}-1\right] \le C(n)d_{\mathcal H}(E,E^{\sharp}).
\]
Using this inequality in \eqref{primo},  we get the claim.
\end{proof}




		
\section{An Isoperimetric inequality}

In \cite{br} the author proved a weighted isoperimetric inequality where the perimeter is replaced by the boundary momentum $W(E)$,  defined as in \eqref{momento}. More precisely, it is proved  that, 
if  $E\subseteq\mathbb{R}^n$ is a Lipschitz set,  then 
\begin{equation}\label{brock}
	\dfrac{W(E)}{V(E)^{\frac{n+1}{n}}}\geq 	\dfrac{W(B)}{V(B)^{\frac{n+1}{n}}}=n\omega_n^{-1/n},
\end{equation}
and  equality holds for a ball.
The inequality  \eqref{brock} implies that, among sets with fixed volume, the boundary momentum and the perimeter are both minimal on balls. 

An isoperimetric inequality for a functional involving the quantities $P(E)$, $W(E)$ and $V(E)$ is proved in  \cite{We} in the planar case and then in \cite{bfnt} in any dimension, restricting to the class of  convex sets.  More precisely, if $E\subseteq\mathbb{R}^n$ is a bounded, open,  convex set, it is proved that 
 \begin{equation}
 \label{Weinstock}
\mathcal{J}(E) =\dfrac{W(E) }{P(E)\;V(E)^{\frac{2}{n}}}\geq  \dfrac{W(B) }{P(B)\;V(B)^{\frac{2}{n}}}=\omega_n^{\frac{-2}{n}}=\mathcal{J}(B)
 \end{equation}
 where equality holds only on  balls centered at the origin. 

In the same spirit, if $F\subset \R^n$ is a bounded,  open, convex set, we define the following functional
\begin{equation} 
\label{nuovo_fun}
    I(F)=\frac{W(F)}{V(F)P(F)^{\frac{1}{n-1}}}.
\end{equation}
The following   isoperimetric inequality holds.
\begin{prop}
Let $n \geq 2$. For every bounded, open, convex  set $F\subset \R^n$,  it holds
\begin{equation}
\label{iso_new} 
I(F) \ge \frac{n}{(n\omega_n)^{\frac{1}{n-1}}}=I(B).
\end{equation}
Equality holds only for balls centered at the origin.
\end{prop}
\begin{proof}
The proof follows easily by using  inequality \eqref{Weinstock}, the standard isoperimetric inequality and  observing that
\[
I(F)=\mathcal{J}(F) \displaystyle\left( \frac{P(F)}{V(F)^{1-\frac{1}{n}}}\right)^{\frac{n-2}{n-1}}. \]
\end{proof}
Our aim  is to prove   a  quantitative version of \eqref{iso_new}. From now on,  we will use the following notation
\begin{equation}
\label{def}
\mathcal{D}(E)=I(E)-\frac{n}{(n\omega_n)^{\frac{1}{n-1}}}=I(E)-I(B).
\end{equation}

%



\subsection{Stability for  nearly spherical sets}
Following  Fuglede's approach (see \cite{f}), we first prove a quantitative version of \eqref{iso_new}  for  nearly spherical sets of the form \eqref{nearlysp}, when $n \ge 3$.  
\begin{thm} \label{fuglede}
Let $n\geq 3$ and $B$ the unit ball of $\R^n$ centered at the origin.  Then,  there exist three positive constants $C_1(n)$, $C_{2}(n)$ and  $\varepsilon=\varepsilon(n)$,  such that,
if  $E \subseteq \R^n$ is  a nearly spherical set with $P(E)=P(B)$ and  $||v||_{W^{1,\infty}} \le \varepsilon$, then
\begin{equation} \label{fuglede1}
  C_1(n) ||v||_{W^{1,1}(\S^{n-1})} \ge \mathcal D(E)
\geq C_{2}(n) ||v||^2_{W^{1,2}(\S^{n-1})}.
\end{equation}
\end{thm}
\begin{proof}
Setting  $v=tu$, with $||u||_{W^{1,\infty}}=1/2$, we have
$||v||_{W^{1,\infty}} = t||u||_{W^{1,\infty}}= t/2 $. Thus, using the expressions of $P(E)$ and $W(E)$  given in \eqref{perimeter_ns} and \eqref{momentum_ns},  we get
\begin{align} \label{step1}
\displaystyle \mathcal D(E) =\frac{n}{P(B)^{\frac{1}{n-1}}}\left( \frac{\ds\int_{\S^{n-1}}\left( 1+tu(x)\right)^n\sqrt{(1+tu(x))^2+t^2|D_{\tau}u(x)|^{2}}\,d\H^{n-1}}{\ds\int_{\S^{n-1}}(1+tu(x))^{n}\,d\H^{n-1}}-1\right)\\ \nonumber
    =\frac{n}{P(B)^{\frac{1}{n-1}}}\left(  \ds\frac{\ds\int_{\S^{n-1}}\left( 1+tu(x)\right)^n\left(\sqrt{(1+tu(x))^2+t^2|D_{\tau}u(x)|^{2}}-1\right)\,d\H^{n-1}}{n V(E)} \right).
\end{align}
Now we prove the lower bound in \eqref{fuglede1}.
Firstly we  take into account the numerator in \eqref{step1}.
Let $f_k(t)=(1+tu)^k\sqrt{(1+tu)^2+t^2|D_\tau u|^2}$. An elementary calculation shows that
\begin{align} \label{tay}
\nonumber
f_k(0)=0, \qquad f_k'(0)=(k+1)u, \qquad f_k''(0)= (k)(k+1)u^2+|D_\tau u|^2 \\
f_k'''(\tau) \le 2(k+2)(k+1)k \left( u^3+|u||D_\tau u|^2 \right)
\end{align}
for any $\tau \in (0,t)$. Thus, since the numerator of \eqref{step1} is given by $f_n (t)-(1+tu)^{n}$,  using the Lagrange expression of the remainder term, we can Taylor expand up to the third order, obtaining
\begin{multline}
\label{1pass}
\int_{\S^{n-1}}\left( 1+tu(x)\right)^n\left(\sqrt{(1+tu(x))^2+t^2|D_{\tau}u(x)|^{2}}-1\right)\,d\H^{n-1} \\\geq t\int_{\S^{n-1}} ud\H^{n-1}+nt^2\int_{\S^{n-1}} u^2 d\H^{n-1}+\frac{1}{2}t^{2} \int_{\S^{n-1}} |D_\tau u|^2d\H^{n-1} \\ 
- C(n)\varepsilon t^2 \int_{\S^{n-1}}\left( u^2+|D_\tau u|^2\right) d\H^{n-1}.
\end{multline}
Since  $P(E)=P(B)$, we have
\begin{align}
    \int_{\S^{n-1}}(1+tu(x))^{n-2}\sqrt{(1+tu(x))^2+t^2|D_{\tau}u(x)|^{2}}\,d\H^{n-1}=\int_{\S^{n-1}}1d\H^{n-1}.
\end{align}
Using \eqref{tay} for $f_{n-2}$, we infer
\begin{multline} 
\label{sameper}
    t\int_{\S^{n-1}} ud\H^{n-1} \geq -\frac{n-2}{2} t^2\int_{\S^{n-1}} u^2d\H^{n-1}- \frac{t^2}{2(n-1)}\int_{\S^{n-1}} |D_\tau u|^2d\H^{n-1} 
    \\
    -C_1(n)\varepsilon t^2 \int_{\S^{n-1}}\left( u^2+|D_\tau u|^2\right) d\H^{n-1}.
\end{multline}
Since $n \ge 3$ ,  using inequality \eqref{sameper}  in \eqref{1pass}, we get
\begin{multline} \label{dis}
    \int_{\S^{n-1}}\left( 1+tu(x)\right)^n\left(\sqrt{(1+tu(x))^2+t^2|D_{\tau}u(x)|^{2}}-1\right)\,d\H^{n-1} \\\geq  \left( \frac{n+2}{2} -C_{2}(n) \varepsilon \right)t^2 \int_{\S^{n-1}} u^2 d\H^{n-1}+\left( \frac{n-2}{2(n-1)} -C_1 \varepsilon \right)t^2 \int_{\S^{n-1}} |D_\tau u|^2d\H^{n-1}.
\end{multline}
Choosing $\varepsilon = \frac{1}{2}\min 
\left\{ \frac{n+2}{2C_{2}(n)}, \frac{n-2}{2C_1 (n-1)} \right\}$, we obtain
$$
\mathcal D(E) \geq C_2(n)  ||tu||^2_{W^{1,2}(\S^{n-1})}\geq C_2(n)  ||v||^2_{W^{1,2}(\S^{n-1})},$$
which is the lower bound in \eqref{fuglede1}.
 Then,
	\begin{multline}
	\dfrac{W(E)}{n V(E)}-1= \dfrac{\displaystyle\int_{\mathbb{S}^{n-1}}\left(1+v(x)\right)^n\left(\sqrt{\left(1+v(x)\right)^2+|D_{\tau}v(x)|^2}-1\right)\; d\mathcal{H}^{n-1}     }{n V(E)}\\ \le C(n) \dfrac{\displaystyle\int_{\mathbb{S}^{n-1}} \left(\sqrt{\left(1+v(x)\right)^2+|D_{\tau}v(x)|^2}-1\right)\; d\mathcal{H}^{n-1}     }{n V(E)}\\  \le C(n)\dfrac{\displaystyle\int_{\mathbb{S}^{n-1}} \left(\sqrt{\left(1+v(x)+|D_{\tau}v(x)|\right)^2}-1\right)\; d\mathcal{H}^{n-1}     }{n V(E)} \\ \le C(n)\dfrac{\displaystyle\int_{\mathbb{S}^{n-1}} \left(|v(x)|+|D_{\tau}v(x)|\right)\; d\mathcal{H}^{n-1}     }{n V(E)} \le C(n)||v||_{W^{1,1}(\S^{n-1})},
	\end{multline}
	where last inequality follows from H\"older inequality and from the following estimate
	\[
	nV(E)= \int_{\mathbb{S}^{n-1}}\left(1+v(x)\right)^n\, d\mathcal{H}^{n-1} \ge n\omega_n\left(\frac{1}{2}\right)^n.
	\]

 \end{proof}
 
\subsection{Stability for convex sets}
Before completing  the proof of the quantitative version of the inequality \eqref{iso_new}, we need the following useful technical lemmas. 
\begin{lem} \label{bound}
Let $n\ge2$. There exists $M>0$ such that, if $F \subseteq \R^n$ is a bounded, open, convex set with $I(F)\le \displaystyle\frac{2n}{(n\omega_n)^{\frac{1}{n-1}}}$ and $|F|=1$, then $F \subset Q_M$, where $Q_M$ is the hypercube centered at the origin with edge $M$.
\end{lem}
\begin{proof}
Since the functional is scale invariant, we can assume $|F|=1$.
Let $L>1$, we have 
\begin{align*}
    W(F)&=\int_{\pa F}|x|^{2}d\H^{n-1}=\int_{(\pa F ) \cap Q_L}|x|^{2}d\H^{n-1}+\int_{\pa F \setminus Q_L}|x|^{2}d\H^{n-1} \\ &\geq \int_{\pa F\cap Q_L}|x|^{2}d\H^{n-1}+L^2P(F; C(Q_L)),
\end{align*} 
where by  $C(Q_L)$ we denote  the complementary set of $Q_L$ in $\R^{n}$.
Since $F$ is convex, also $F \cap Q_L$ is convex  and then
\begin{equation}\label{bound1}
P(F) \le  P(F; C(Q_L))+ P(F_; Q_L)\le P(F; C(Q_L))+2L^{n-1},    
\end{equation}
by the monotonicity of the perimeter. Suppose $P(F)>L^n$; then,  equation \eqref{bound1} gives $ P(F; C(Q_L) )\geq L^n-2L^{n-1}$ and, as a consequence, 
\begin{equation}\label{bound2}
I(F) \geq \frac{\ds\int_{\pa F\cap Q_L}|x|^{2}d\H^{n-1}+L^2P(F; C(Q_L))}{\left(  P(F; C(Q_L)+2L^{n-1} \right)^{\frac{1}{n-1}}}>\frac{L^{n+2}-L^{n+1}}{L^{\frac{n}{n-1}}}.    
\end{equation}
The previous inequality leads to a contradiction for  $L$ large enough, since we are assuming  $I(F)<\displaystyle\frac{2n}{(n\omega_n)^{\frac{1}{n-1}}}$, while the last term of the above inequality diverges when $L \rightarrow \infty$.
Thus, there exists $L_0$ such that, for every convex set $F$ with $I(F)\le \displaystyle\frac{2n}{(n\omega_n)^{\frac{1}{n-1}}}$,  we have $P(F) < L_0^n$. Since   $|F|=1$ and  $P(F)\le L_0^n$, using   \eqref{diam},  we get
$$
\text{diam} (F) \le C(n) L_0^{n(n-1)}.
$$
The last inequality  proves \eqref{bound1}, if we choose $M=C(n)L_0^{n(n-1)}$.
\end{proof}
\begin{lem}
\label{low}
Let $(F_j)\subseteq \R^n$, $n \ge 2$, be a sequence of convex sets such that  $I(F_j) \le \displaystyle\frac{2n}{(n\omega_n)^{\frac{1}{n-1}}}$ and $P(F_j)=P(B)$. Then, there exists a convex set $F\subseteq \R^n$ with $P(F)=P(B)$ and such that, up to a  subsequence, 
\begin{equation} \label{comp}
    |F_j\Delta F|\rightarrow 0 \,\,\,\ \text{and} \quad I(F) \le \liminf I(F_j).
\end{equation}
\begin{proof}
The existence of the limit set $F$ comes from the proof of Lemma \ref{bound}: since ${I(F_j) < \displaystyle\frac{2n}{(n\omega_n)^{\frac{1}{n-1}}}}$,  there exists $M>0$ such that $F_j \subset Q_M$ and $P(F_j)=P(B) $ for every $i \in \mathbb{N}.$ Thus,  the sequence $\{ \Chi{E_j} \}_{j\in \mathbb{N}}$ is precompact in $BV(Q_M)$ and so there exists a subsequence and a set $F$ such that $|F\Delta F_j|\rightarrow 0$. Moreover, from Lemma \ref{periconv}, we have that $P(F)=P(B)$.
Note that we can write 
$$
W(F)= \sup \left\{ \int_F \text{div} \left( |x|^2 \phi(x)  \right)dx,\,\,\,\,\, \phi \in C^1_c(Q_M, \R^n), \,\,\,\,\, ||\phi||_\infty \le 1 \right\}.
$$
Observing that 
$$
\int_F |\text{div} \left( |x|^2 \phi(x)  \right)|dx \le M||\text{div} \phi ||_\infty +M^2,
$$
using the dominate convergence theorem, we have that the functional
$$
F \rightarrow \int_F \text{div} \left( |x|^2 \phi(x)  \right)dx
$$
is continuous with respect to the $L^1$ convergence. Hence, since $W(F)$ is obtained by taking the supremum of continuous functionals, it is lower semicontinuous. As a consequence, we obtain  the inequality  \eqref{comp}.
\end{proof}
\end{lem} 
The next result allows us to reduce the study of  the stability issue to  nearly spherical sets. 
\begin{lem}
\label{3.8}
Let $n \ge 2$. For every $\varepsilon>0$, there exists $\delta_{\varepsilon}>0$ such that,  if  $E\subseteq \R^n$ is a bounded, open, convex set  with $P(E)=P(B)$ and $\mathcal{D}(E)< \delta_{\varepsilon}$, with  $\mathcal{D}(E)$ defined as in \eqref{def}, then  there exists a Lipschitz function $v \in W^{1,\infty}(\mathbb S^{n-1})$ such that $E$ is a nearly spherical set parametrized by $v$ and $\|v\|_{W^{1,\infty}} \le \varepsilon.$
\end{lem}
\begin{proof} 
Firstly, we  prove that $d_\H(E,B) < \varepsilon$.
Suppose by contraddiction that there exists $\varepsilon_0>0$  such that, for every $j \in \mathbb{N}$, there   exists a convex set $E_j$ with $I(E_j)-\displaystyle\frac{2n}{(n\omega_n)^{\frac{1}{n-1}}}<\frac{1}{j}$,  $d_\H(E_j,B)>\varepsilon_0$ and $P(E_j)=P(B)$. By Lemma \ref{low},  we have that there exists a convex set $E$  such that
$E_j$ converges to $E$ in measure and $P(E)=P(B)$. From the semicontinuouity of 
$W(E)$, we have that $I(E) \le \liminf I(E_j) \le \displaystyle\frac{2n}{(n\omega_n)^{\frac{1}{n-1}}}$. Since $B$ is the only minimizer of the functional $I$, we obtain the contradiction. Then,  since  $E$ is convex and $d_\H(E,B) \le \varepsilon$, $E$ contains the origin and so there exists a Lipschitz function $v \in L^{\infty}(\S^{n-1})$, with $||v||_\infty< \varepsilon$, such that
$$
\pa E=\{ x(1+v(x)), \, x\in \S^{n-1} \}.
$$
Now, in order to complete the proof,  we have only  to show that $\|v\|_{W^{1,\infty}}$ is small when $\mathcal D(E )$ is small. This is a consequence of Lemma \ref{utile}.
\end{proof}
Now we can prove the  stability result for the  inequality \eqref{iso_new}. We first consider the case  $n \ge 3$. The two dimensional case will be discussed separately in the next section.
\begin{thm}\label{main2}
Let $n \geq 3$. There exist $\delta$ and  $C(n)>0$ such that, if $E\subseteq \R^n$ is a  bounded,  open,   convex set with  $\mathcal D(E)\le \delta$,  then
\begin{eqnarray} \label{quantisharp}
\displaystyle \mathcal{A_\mathcal{H}}(E) \le 
\begin{cases}
\sqrt{  \mathcal D(E)\log\frac{1}{\mathcal D(E)}} & n=3
\\[.2cm]
C(n)  \left(\mathcal D(E) \right)^{\frac{2}{n+1}} & n\ge 4,
\end{cases}
\end{eqnarray}
where $\mathcal{A_\mathcal{H}}(E) $  and $\mathcal D(E)$ are  defined  in  \eqref{asy} and \eqref{def} respectively.
\end{thm}
\begin{proof}
Since the functional $I$ is scaling invariant,  we can suppose that $E$ is a  convex set  with  $P(E)=P(B)$.
We fix now $\varepsilon >0$.
%
Using Lemma \ref{3.8},  we can  suppose  that 
there exists $v \in W_{1,\infty}(\S^{n-1})$ with $||v||_{W_{1,\infty}}< \varepsilon$ such that
$$
\pa E=\{ x(1+v(x)), \, x\in \S^{n-1} \}.
$$
Then,  if we take $ \varepsilon$  small enough, by Theorem \ref{fuglede},  we obtain
$$
\mathcal D(E)\geq C(n)\displaystyle ||v||^2_{W^{1,2}(\S^{n-1})}.
$$ 
Let $F=\lambda E$, with $\lambda$ such that $V(F)=V(B)$.
From the isoperimetric inequality, it follows that $\lambda >1$.
Since the quantity $I(E)$ is scaling invariant, we have that  $I(F)=I(E)$ and, from the definition of $F$, that
\begin{equation}
    \pa F= \{ \lambda x(1+ v(x)), x \in \S^{n-1} \} = \{  x(1+(\lambda -1 + \lambda v(x))), x \in \S^{n-1} \}.
\end{equation}
Using the definition of $\lambda$ , we obtain
$$
\lambda^n-1= \frac{V(B)}{V(E)}-1= \frac{\sum_{k=1}^n \binom{n}{k} \displaystyle\int_{\S^{n-1}}v^k\H^{k-1}}{V(E)}
$$
and, as a consequence, 
\begin{equation}
    \lambda -1= \frac{\sum_{k=1}^n \binom{n}{k} \displaystyle\int_{\S^{n-1}}v^k\H^{k-1}}{V(E)\sum_0^{n-1}\lambda^k}.
\end{equation}
Let now $h(x)=\lambda -1 +\lambda v(x)$. Note that $||h||_{W^{1,\infty}}< 2^n||v||_{W^{1,\infty}}$ and that $\lambda^{n}\in (1,2)$. Moreover, using H\"older inequality, it is easy to  check that
$$
||h||^2_{L^{2}(\S^{n-1})} \le 2^{n+2}||v||^2_{L^{2}(\S^{n-1})} \quad \text{ and }\quad ||D_\tau h||^2_{L^{2}(\S^{n-1})}\le2^{1/n}
||D_\tau v||^2_{L^{2}(\S^{n-1})}.
$$
Thus,
\begin{equation}
    \mathcal D(F)=\mathcal D(E) \geq C_2(n) ||h||^2_{W^{1,2}(\S^{n-1})} \geq 2^{-n-1}C_2(n) ||h||^2_{W^{1,2}(\S^{n-1})} .
\end{equation}
Let $g= (1+h)^n-1$. Then, since $V(F)=V(B)$, we have $\int_{\S^{n-1}} gd\H^{n-1}=0$ and, from the smallness assumption on $u$, we immediately have  $\frac{1}{2}|h|\le  |g| \le 2|h|$ and $\frac{1}{2}|Dh|\le  |Dg| \le 2|Dh|$. Now we have to distinguish the cases $n=3$ and $n \ge 4$ , since we are going  to apply the interpolation  Lemma \ref{interp} to $g$. In the case $n\geq 4$, we get
\begin{align*}
    ||h||_\infty \le 2 ||g||_\infty \le C(n)||D_\tau g||^{\frac{2}{n-1}}_{L^{2}(\S^{n-1})} ||D_\tau g||_{L^{\infty}(\S^{n-1})}^{\frac{n-3}{n-1}}
    \\
    \le C(n)||D_\tau h||_{L^{2}(\S^{n-1})}^{\frac{2}{n-1}} ||D_\tau h||_{L^{\infty}(\S^{n-1})}^{\frac{n-3}{n-1}}
    \le C(n)||D_\tau h||_{L^{2}(\S^{n-1})}^{\frac{2}{n-1}} ||h||_{L^{\infty}(\S^{n-1})}^{\frac{n-3}{2(n-1)}},
\end{align*}
where in the last inequality we use \eqref{utileq}. From the above chain of inequalities we deduce
$$
||h||_{L^{\infty}} ^{\frac{n+1}{2}} \le C(n) ||D_\tau h||_{L^{2}(\S^{n-1})}^2
$$
and finally, recalling that $F=\lambda E$ and $V(F)=V(B)$, we get
\begin{multline}
\mathcal D(E) \geq C_n||D_\tau h||_{L^{2}(\S^{n-1})}^2 \geq C_n||h||_{L^{\infty}} ^{\frac{n+1}{2}}= C_n d_\H (F,B)^{\frac{n+1}{2}}=C_{n}\left( \frac{d_\H (E,E^{\sharp})}{V(E)^\frac{1}{n}}\right)^{\frac{n+1}{2}}.
\end{multline}
So, using  \eqref{distsharp} and  the isoperimetric inequality, we obtain the desired result \eqref{quantisharp} in the case $n\geq 4$.  We proceed in an analogous way in the case $n=3$.
Firstly we observe that, by definition of $h$, there exists a positive constant $A$ such that $||v ||_{W^{1,1}(\mathbb{S}^{n-1})}\leq A ||h||_{W^{1,1}(\mathbb{S}^{n-1})}.$ 
Then,  the upper bound in \eqref{quantisharp} in terms of $h$, can be written as follows\begin{equation}\label{Wh}
\mathcal{D}(E)=\mathcal{D}(F)\leq \bar{C}||h||_{W^{1,1}(\mathbb{S}^{n-1})},
\end{equation}
with $\bar{C}$ positive costant depending on the dimension.
 Applying Lemma \ref{interp} to $g$ and using Lemma \ref{utile}, we obtain:
	\begin{multline*}
	|| h||^2_\infty\leq 4 || g||^2_\infty\leq 16 ||D_{\tau}g||^2_{L^2(\mathbb{S}^{n-1})}\log\left[ \dfrac{8e||D_{\tau}g ||_\infty^2}{||D_{\tau} g||^2 _{ L^2(\mathbb{S}^{n-1}) }}   \right]\\\leq 64 ||D_{\tau}h||^2_{L^2(\mathbb{S}^{n-1})}\log\left[ \dfrac{32e||D_{\tau}h ||_\infty^2}{||D_{\tau} g||^2 _{ L^2(\mathbb{S}^{n-1}) }}   \right]\leq 64 ||D_{\tau}h||^2_{L^2(\mathbb{S}^{n-1})}\log\left[ \dfrac{C \;||h ||_\infty}{||D_{\tau} g||^2 _{ L^2(\mathbb{S}^{n-1}) }}   \right].
	\end{multline*}
	Choosing now $||h||_\infty$ small enough, from the upper bound in  \eqref{fuglede1}, we have 
	\begin{equation}\label{first}
	|| h||^2_{\infty}\leq 64 ||Dh||^2_{L^2(\mathbb{S}^{n-1})} \log\left[\dfrac{1}{\mathcal{D}(E)}\right],
	\end{equation}
	and, as a consquence, using \eqref{fuglede1} and \eqref{first},
	\begin{equation*}
	\mathcal{D}(E)\log\left(\dfrac{1}{\mathcal{D}(E)}\right)\geq C_1(n) ||D_{\tau} h||_{L^2(\mathcal{S}^{n-1})}\log\left(\dfrac{1}{\mathcal{D}(E)}\right)\geq C ||h||^2_\infty \dfrac{\log\left(\dfrac{1}{\mathcal{D}(E)}\right)}{\log\left(\dfrac{1}{\mathcal{D}(E)}\right)}=C ||h||^2_\infty.
	\end{equation*}
\end{proof}\subsection{Optimality issue.}		

In this section we will show the sharpness of inequality \eqref{quantisharp} and, as a consequence,  the sharpness for  the exponent of the quantitative Weinstock inequality. We start by taking into exam the case  $n=3$.
\begin{thm}\label{sharpness3}
	Let $n= 3$. There exists a family of convex sets $\{E_\alpha \}_{\alpha>0}$ such that for every $\alpha$
	\begin{equation*}
	\mathcal{	D}(E_\alpha )\rightarrow 0,\quad {\rm when } \;\alpha\rightarrow 0
	\end{equation*}
	and 
	\begin{equation}\label{sharpy}
	\mathcal{A}_{\mathcal{H}}(E_\alpha)= C \sqrt{  \mathcal D(E_\alpha)\log\frac{1}{\mathcal D(E_\alpha)}}
	\end{equation}
	where $C$ is a suitable positive constant independent of $\alpha$.
\end{thm}
\begin{proof}
	We follow the idea contained in \cite{f} (Example $3.1$) and  recall it here for the convenience of the reader.  Let $\alpha\in(0,\pi/2)$ and consider the following function $\omega=\omega(\varphi)$ defined over $\mathbb{S}^2$ and depending only on the spherical distance $\varphi$, with $\varphi\in[0,\pi]$, from a prescribed north pole $\xi^*\in\mathbb{S}^2$:
	\begin{eqnarray}
	\omega=\omega(\varphi)=
	\begin{cases}
	
	-\sin^2\alpha \log\left(\sin \alpha\right)+\sin\alpha\left(\sin\alpha-\sin\varphi\right)&{\rm for}\; \sin\varphi\leq \sin\alpha
	\\[.2cm]
	-\sin^2(\alpha)\log\left(\sin\varphi\right)\,\, & {\rm for}\; \sin\varphi\geq \sin\alpha.
	\end{cases}
	\end{eqnarray}
	Let $g:=\omega-\bar{\omega}$, with $\bar{\omega}$ the mean value of $\omega$, i.e. 
	\begin{equation*}
	\bar{\omega}=\int_{0}^{\pi/2} \omega(\varphi)\sin\varphi\;d\varphi =\left(1-\log 2\right)\alpha^2+O(\alpha^3), 
	\end{equation*}
	when $\alpha$ goes to  $0$, and let
	$$R:=\left( 1+3g  \right)^{1/3}=1+h. $$
	The  $C^1$ function $R=R(\varphi)$ determines  in polar coordinates $(R,\varphi) $ a planar curve. We rotate this curve about the line $\xi^*\mathbb{R}$, determining in this way the boundary of a convex and bounded set, that we call $E_\alpha$.
	We can observe that $h$ and $g$ are the same fuctions cointained in the proof of Theorem \ref{main2}. The set $E_\alpha$ is indeed a nearly spherical set, which has $h$ as a representative function and with  $V(E_\alpha)=V(B)$.
	Therefore, taking into account the computations contained in the proof of Theorem \ref{main2} relative to the functions $h$ and $g$ and the ones contained in  \cite{f} combined with \eqref{fuglede1}, we have
	\begin{equation}
	||g||_\infty=\alpha^2\log\dfrac{1}{\alpha}+O(\alpha^2),
	\end{equation}
	\begin{equation}\label{sharpy1}
	||h||_\infty\geq\dfrac{1}{2}||g||_\infty=\dfrac{1}{2}\alpha^2\log\dfrac{1}{\alpha}+O(\alpha^2),
	\end{equation}
	and
	\begin{equation*}
	|| \nabla h||^2_2=||\nabla g ||^2_2=\alpha^4\log\left(\dfrac{1}{\alpha}\right)+O(\alpha^4).
	\end{equation*}
Using \eqref{Wh}, we obtain:
	\begin{equation}
	\mathcal{D}(E_\alpha)=O\left(\alpha^4\log\dfrac{1}{\alpha}\right)
	\end{equation}
Consequently
	\begin{equation} \label{deficitsharp}
	\mathcal{D}(E_\alpha)\log\left( \dfrac{1}{\mathcal{D}(E_\alpha)}\right)=O\left(\alpha^2\log\dfrac{1}{\alpha}\right)^2.
	\end{equation}
So, we have that $\mathcal{D}(D_\alpha)\rightarrow0$ as $\alpha$ goes to $0$ and, combining \eqref{sharpy1} with \eqref{deficitsharp}, we have  the validity of \eqref{sharpy}.
\end{proof}
We show now the sharpness of the quantitative Weinstock inequality in dimension $n\geq4$.
\begin{thm}\label{sharpness4}
	Let $n\geq 4$. There exists a family of convex sets $\{P_\alpha \}_{\alpha>0}$ such that 
	\begin{equation*}
	\mathcal{	D}(P_\alpha )\rightarrow 0,\quad {\rm when } \;\alpha\rightarrow 0
	\end{equation*}
	and 
	\begin{equation*}
	\mathcal{A_{\mathcal{H}}}(P_\alpha)\geq C(n) \left( \mathcal{D}(P_\alpha)\right)^{2/(n+1)},
	\end{equation*}
	where $C(n)$ is a suitable positive  constant. 
\end{thm}
\begin{proof}
	In this proof we follow the costruction given in \cite{f} (Example $3.2$). Let $\alpha\in]0,\pi/2[$ and let  $P_\alpha$ be the convex hull of $B\cup\{-p,p\}$, where $p\in\mathbb{R}^n$ is given by
	$$|p|=\dfrac{1}{\cos \alpha}. $$ 
	We have that
	\begin{equation*}\label{volumesharp}
	V(P_\alpha)=\omega_n+\dfrac{2}{n(n+1)}\omega_{n-1}\alpha^{n+1}+O(\alpha^{n+3})
	\end{equation*}
	and 
	\begin{equation*}\label{surfacesharp}
	P(P_\alpha)=n V(P_\alpha).
	\end{equation*}
	We provide here the computation of the boundary momentum, that is
	\begin{equation}
	W(P_\alpha)=\dfrac{2\omega_{n-1}}{n(n+1)}\dfrac{\left(\sin(\alpha)\right)^{(n-1)}}{\cos(\alpha)}    \left(n^2+n+2\tan^2(\alpha)\right)+2(n-1)\left[\dfrac{\sqrt{\pi}\;\Gamma\left(\dfrac{n-1}{2}\right)}{2\Gamma\left(\dfrac{n}{2}\right)}-\int_{0}^{\alpha}\sin^{n-2}(\theta)\;d\theta\right].                              
	\end{equation}
As a consequence, we have 
	\begin{equation*}
	(n\omega_n)^\frac{1}{n-1}V(P_\alpha)P(P_\alpha)^\frac{1}{n-1}\mathcal{D}(P_\alpha)=\left(n\omega_n\right)^{\frac{1}{n-1}} \dfrac{2\omega_{n-1}}{n+1}\dfrac{(n-2)}{n (n-1)}\alpha^{n+1}+o(\alpha^{n+3}).
	\end{equation*}
	Since $\mathcal{A_{\mathcal{H}}}(D_\alpha)$ behaves asimptotically as $\alpha^2$, we have proved the desired claim.
\end{proof}
\section{The planar case}
In this section we discuss the stability of  the isoperimetric  inequality \eqref{iso_new} in the plane. In $\R^2$ we cannot use the same arguments   used in higher dimensions to obtain a stability result for \eqref{iso_new}. Moreover, we observe that,  in   two dimension, the  inequality \eqref{Weinstock} contained in \cite{bfnt}  and  the inequality \eqref{iso_new} are proved by  Weinstock in \cite{We}, using   the representation of a two dimensional convex set via its support function.  More precisely, let $E \subset \R^2$ be an open, smooth, convex set in the plane containing the origin and let $h(\theta)$ be the support function of $E$ with $\theta \in [0,2 \pi]$.   Weinstock proved in \cite{We} the following inequality (see also  \cite{bfnt} for  details)  
\begin{equation}
\label{we_st}
	\pi W(E)-P(E)V(E)\geq\frac{P(E)}{2}\int_{0}^{2\pi}p^2(\theta)\:d\theta \ge 0,
	\end{equation}
	where, for every $\theta\in[0,2\pi]$,  $p(x)$ is defined by	\begin{equation*}
	h(\theta)=\frac{P(E)}{2\pi}+p(\theta).
	\end{equation*}
	By the definition of support function,  it holds
 \begin{gather}
 \label{supp}
 \int_{0}^{2\pi} h(\theta)\;d\theta= P(E).
 \end{gather} 	
 Moreover, since  $E$ is convex, we have 
 \begin{gather}
 \label{supp2}
h(\theta)+h''(\theta)\ge0 .
 \end{gather} 
 Then, the function $p$ verifies 
 $$ \int_{0}^{2\pi} p(\theta)\;d\theta=0,$$
 and 
  \begin{gather}
 \label{supp3}
\frac{P(E)}{2\pi}+p(\theta)+p''(\theta)\ge0 .
 \end{gather}
 We observe that 
\begin{equation}
\label{linf}
\|p\|_{L^\infty([0,2\pi])}=d_{\mathcal H}(E,E^*),
\end{equation}
where $E^*$ is the disc centered at the origin having the same perimeter as $E$.
 Consider $\theta_{0} \in [0,2\pi]$ such that $\|p\|_{L^{\infty}}=p(\theta_{0})$.
By using property \eqref{supp3}, it is not difficult to prove the following result.
 \begin{prop}
 \label{2dim}
 Let $p$ be as above,  then 
 \begin{equation}
 \label{par1}
  p(\theta)\ge\gamma(\theta),
 \end{equation}
 where $\gamma(\theta):=p(\theta_0)-\displaystyle\frac{1}{2}\displaystyle \left(\frac{P(E)}{2\pi}+p(\theta_{0}))\right)\left(\theta-\theta_{0}\right)^{2}$ is a parabola which vanishes at the following points 
 \[ 
	 \theta_{1,2} =\theta_{0}\pm\sqrt{\displaystyle \frac{2p(\theta_0)}{\frac{P(E)}{2\pi}+p(\theta_{0})}}.
	 \]  
 \end{prop}
 \begin{proof}
 By property \eqref{supp3}, we obtain
	 \begin{multline}
	 p(\theta)=p(\theta_{0})+\int_{\theta_{0}}^{\theta} p'(t) \,dt = p(\theta_{0})+\int_{\theta_{0}}^{\theta}\int_{\theta_{0}}^{t} p''(s) \,ds \, dt \\\ge p(\theta_{0})+\int_{\theta_{0}}^{\theta}\int_{\theta_{0}}^{t} -\left( \frac{P(E)}{2\pi}+p(s)\right) \,ds \,dt\\\ge p(\theta_{0})- \left( \frac{P(E)}{2\pi}+p(\theta_{0})\right) \frac{\left(\theta-\theta_{0}\right)^{2}}{2},
	 \end{multline}
	  which is the claim. 
	Then, $p$ is above the parabola $\gamma$,  that attains its zeros at the following points:
	\[
	 \theta_{1,2} =\theta_{0}\pm\sqrt{\displaystyle \frac{2p(\theta_0)}{\frac{P(E)}{2\pi}+p(\theta_{0})}}.
	 \]
This concludes the proof.
 \end{proof}
	Inequality \eqref{we_st} implies Weinstock inequality but it  hides also a stability result. Indeed, by using the previous Proposition,   we get the following quantitative Weinstock inequality in the plane.
	\begin{thm}
	\label{two}
	There exist $\delta$ and a positive constant $C$ such that, if  $E \subset \R^2$ is a bounded,  open, convex with $\mathcal D(E)\le \delta$, then
	  \[ C\mathcal{A}_{\mathcal{H}}(E)^{\frac{5}{2}} \le  \mathcal D(E),\]
	  where $\mathcal{A_\mathcal{H}}(E) $  and $\mathcal D(E)$ are  defined  in  \eqref{asy} and \eqref{def} respectively.
	  Moreover,  the exponent $\frac{5}{2}$ is sharp.
	\end{thm}
	\begin{proof}
Since the functional $\mathcal{D}$ is scaling invariant,  we can assume that $E$ is a strictly convex set of  finite measure with  $P(E)=P(B)=2\pi$. From Lemma \ref{3.8}, if we take a sufficiently small $\varepsilon$, there exists $\delta>0$ such that, if  $\mathcal{D}(E)\le \delta$, then $E$ contains the origin, its boundary can be parametrized  as above by means the support function and, by \eqref{linf}, 
	\[
	d:=\|p\|_{L^{\infty}([0,2\pi])} \le \varepsilon.
	\]
Under these assumptions, since  in particular $|d|<\frac{1}{2}$,  Proposition \ref{2dim} gives
\begin{equation}
\label{app2}
p(\theta) \ge d-\left(\frac{1+d}{2}\right)(\theta-\theta_0)^2 \ge d -\frac{(\theta-\theta_0)^2
}{4}.\end{equation}
	Denoting  by $\theta_{1,2}$ the zeros of the parabola $d-\frac{(\theta-\theta_0)^2
}{4}$, that are
\[
\theta_{1,2}=\theta_0\pm2\sqrt{d},
\]
by using  \eqref{we_st}, the isoperimetric inequality,  H\"older inequality and  \eqref{app2},   we get
	\begin{multline}
	\mathcal D(E)=\frac{W(E)}{P(E)V(E)}-\frac{1}{\pi}=\displaystyle \frac{\pi W(E)-P(E)V(E)}{\pi P(E)V(E)} \ge \frac{1}{2\pi^2 }\int_{0}^{2\pi}p^2(\theta)\,d\theta \\ >\frac{1}{2\pi^2}\int_{\theta_{1}}^{\theta_{2}}p^2(\theta)\,d\theta\ge\frac{1}{2\pi^2(\theta_{2} -\theta_{1})}\left(\int_{\theta_{1}}^{\theta_{2}}p(\theta)\,d\theta\right)^{2}>\frac{16}{9\pi^2}d^{\frac{5}{2}}.
	\end{multline}
By \eqref{linf} and \eqref{asy}, being $P(E)=2\pi$, we get the claim.	

				In order to conclude the proof, we have to show the sharpness of the exponent. We construct a family of strictly convex sets $E_{\varepsilon}$, with $P(E_{\varepsilon})=2\pi$,  such that 
\[ \mathcal{D}(E_{\varepsilon}) \to 0 \text{ for } \varepsilon \to 0,\]
and
\[
\mathcal{A}_{\mathcal{H}}(E_{\varepsilon})= C \varepsilon^{\frac{5}{2}}.
\]
Let us consider the convex set $E$ having the following support function:
		\[h(\theta) = 1+p(\theta), \quad \theta \in [0,2\pi],\]		
where the function $p$ is the following
\[
p(\theta)=\begin{cases}
-b &\text{if } \theta \in [0,\alpha_1]\\
\varepsilon-\frac{(\theta-\pi)^2}{4} &\text{if } \theta \in[\alpha_1,\alpha_2]\\
-b &\text{if } \theta \in [\alpha_2,2\pi].
\end{cases} 
\]
Here the parameters $b$, $\alpha_1$ and $\alpha_2$ are
\[
 b=\frac{4}{3}\displaystyle \frac{\varepsilon^{\frac{3}{2}}}{\pi-2\sqrt{\varepsilon}}, \quad \alpha_1=\pi-2\sqrt{\varepsilon}, \quad \alpha_2=\pi+2\sqrt{\varepsilon}.
\]
By construction, we have that 
\[
P(E_{\varepsilon})= 2\pi \,\,\,\text{ and } \,\,\,\int_0^{2\pi}p(\theta)\, d\theta=0.
\]
We  recall that (see for instance \cite{We,We2,bfnt})
\[\begin{cases}
V(E_{\varepsilon})=\displaystyle\frac{1}{2}\displaystyle\int_0^{2\pi}\left(h^2(\theta)+h(\theta)h''(\theta)\right)\,d\theta \\[.2cm]
W(E_{\varepsilon})=\displaystyle\int_0^{2\pi}\left(h^3(\theta)+\frac{1}{2}h^2(\theta)h''(\theta)\right)\,d\theta.
\end{cases}
\]
Arguing as in the proof of  Weinstock inequality, a simple calculation gives
\begin{multline}
\pi W(E)-P(E)V(E)=\pi \displaystyle\int_0^{2\pi} p^2(\theta) \left(2+p(\theta) + \frac{1}{2}p''(\theta))\right)\,d\theta=\\2\pi \displaystyle\int_0^{2\pi}p^2(\theta)\,d\theta +\pi \displaystyle\int_0^{2\pi}p^3(\theta)\,d\theta+\displaystyle \frac{\pi}{2}\displaystyle\int_0^{2\pi}p^2(\theta)p''(\theta)\, d\theta = C \varepsilon^{\frac{5}{2}}+O(\varepsilon^{3}),
\end{multline}
where $C$ is a positive constant. This concludes the proof.
\end{proof}

\section{Proof of Theorem \ref{main_w}}

The proof is a consequence of  Theorems \ref{main2} and \ref{two}.
Since all the quantities involved are invariant under translations, we can assume that $\partial \Omega$ has the origin as barycenter. 


Under this assumption in  \cite{bfnt} it is proved that
\begin{equation}\label{ineq}
\sigma(\Omega)\leq\dfrac{n V(\Omega)}{W(\Omega)}.
\end{equation}
By \eqref{ball},  it holds
\begin{equation}\label{ball}
	\sigma(B_{R}(x))=\frac{1}{R}=\left[\dfrac{n\omega_n}{P(\Omega)}\right]^{1/(n-1)},
\end{equation}
 then, using the previous inequality and \eqref{nuovo_fun}, we have

\begin{equation}
\label{pripas}
		 \dfrac{\sigma(B_{R}(x))-\sigma(\Omega)}{\sigma(\Omega)}=\dfrac{\sigma(B_{R}(x))}{\sigma(\Omega)}-1\geq \dfrac{W(\Omega)}{n V(\Omega)}\left(\dfrac{n\omega_n}{P(\Omega)}\right)^{1/(n-1)}-1={\frac{(n\omega_n)^{\frac{1}{n-1}}}{n}}\mathcal D(\Omega).
\end{equation}
Let $\delta $  be as in Theorem \ref{main2}. Then if $\Omega$ is such that $\sigma_{1}(B_{R}(x))\le(1+\bar\delta)\sigma_{1}(\Omega)$, with $\bar \delta= \frac{(n\omega_n)^{\frac{1}{n-1}}}{n}\delta $  then  $\mathcal D(\Omega)\le\delta $  and, for $n\ge 4$ from \eqref{quantisharp} in Theorem \ref{main2}, we get 
\[
		 \dfrac{\sigma(B_{R}(x))-\sigma(\Omega)}{\sigma(\Omega)}\ge \\C(n)(\mathcal{A}_{\mathcal{H}}(\Omega))^{\frac{n+1}{2}}. 
\]
 If $n=3$,  we can conclude  a similar way,  observing that  $f(t)= t \log\left(\frac{1}{t}\right)$ is invertible for $0<t<e^{-1}$. Thus, being $D(\Omega)$ small, we can explicit it in \eqref{quantisharp}, obtaining the thesis. The result in two dimension follows from Theorem \ref{two}.

\end{document}